\documentclass[a4paper,11 pt]{amsart}

\usepackage[sc]{mathpazo}
\usepackage{eulervm}
\usepackage{hyperref}

\usepackage[utf8]{inputenc}
\usepackage[T1]{fontenc}

\usepackage[english]{}
\usepackage{mathtools}
\usepackage{braket}
\usepackage{amsmath, amssymb, stmaryrd, mathabx}
\usepackage{enumerate}
\usepackage{tikz-cd}
\usepackage{tikz,float, hyperref}
\usepackage{titlesec}

\titleformat{\section}
{\normalfont\normalsize\bfseries}{\thesection.}{1em}{}
\titleformat{\subsection}
{\normalfont\normalsize\itshape}{\thesubsection.}{1em}{}
\titleformat{\subsubsection}
{\normalfont\normalsize\itshape}{\thesubsubsection.}{1em}{}

\newcommand{\R}{\mathbb{R}}

\newcommand{\ra}{\rightarrow}
\newcommand{\pum}{pure$^-$\ }
\newcommand{\pup}{pure$^+$\  }

\newcommand{\bes}{\begin{eqnarray*}}
	\newcommand{\ees}{\end{eqnarray*}}
\newcommand{\beq}{\begin{eqnarray}}
	\newcommand{\eeq}{\end{eqnarray}}

\theoremstyle{plain}
\newtheorem{theorem}{Theorem}[section]
\newtheorem{proposition}[theorem]{Proposition}
\newtheorem{lemma}[theorem]{Lemma}
\newtheorem{conjecture}{Conjecture}
\newtheorem{corollary}[theorem]{Corollary}
\theoremstyle{definition}
\newtheorem{definition}[theorem]{Definition}
\newtheorem{example}{Example}

\newtheorem{remark}[theorem]{Remark}

\title{Eschers and Stanley's chromatic e-positivity conjecture in length-2}
\author{Alexandre Rok}
\address{Section de mathématiques, Université de Genève}
\email{Alexandre.Rok@unige.ch}
\author{Andras Szenes}
\address{Section de mathématiques, Université de Genève}
\email{Andras.Szenes@unige.ch}
\begin{document}	
	\begin{abstract} 
		We give a short proof of the chromatic e-positivity conjecture of 
		Stanley for length-2 partitions.
	\end{abstract}
	
	\maketitle
	
	\section{Introduction}\label{sec:intro}
	
	Let $G$ be a finite graph, $V(G)$ - the set of vertices of $G$, $E(G)$ - 
	the set of edges of $G$. 
	
	\begin{definition} \label{coloring} A {\em proper coloring} $c$ of  $G$ is
		a map $$c:V\rightarrow\mathbb{N}$$ such that no two adjacent
		vertices are colored in the same color.
	\end{definition}
	
	For each coloring $c$ we define a monomial $$x^c = \prod_{v\in
		V}x_{c(v)},$$ where $x_1, x_2, ..., x_n,...$ are commuting
	variables.  We denote by $\Pi(G)$ the set of all proper colorings
	of $G$, and by  $\Lambda$ the ring of symmetric functions in the infinite 
	set
	of variables $\{x_1, x_2,...\}.$

	In \cite{Stanley95a}, Stanley defined the chromatic symmetric function of a 
	graph as follows.
	
	\begin{definition} \label{chromfunction} The \em chromatic symmetric
		function \normalfont $X_G\in\Lambda$ of a graph $G$ is the sum of the 
		monomials $x^c$
		over all proper colorings of
		$G$: $$X_G=\sum\limits_{c\in\Pi(G)}x^c.$$
	\end{definition} 
		\begin{remark}
We could equally consider a finite, but sufficiently large number colors $ r $. In this case, we would have $ X_G\in\Lambda_r $, the set of symmetric polynomials in $ r $ variables. 
This would somewhat complicate the notation, but would not change any of our results.
\end{remark}
	
	\begin{definition} \label{efunc} Denote by $e_m$ the $m$-th elementary
		symmetric function:
		$$e_m = \sum\limits_{i_1<i_2<...<i_m}x_{i_1}\cdot
		x_{i_2}\cdot...\cdot x_{i_m},$$
		where $i_1,..,i_k\in \mathbb{N}$.  Given a non-increasing sequence of
		positive integers (we will call these {\em partitions})
		$$\lambda = (\lambda_1\geq \lambda_2\geq...\geq\lambda_k),\ \lambda_i\in
		\mathbb{N},$$
		we define the \textit{elementary symmetric} function
		$e_{\lambda} = \prod\limits_{i=1}^k e_{\lambda_i}.$ These functions
		form a basis of $\Lambda.$
	\end{definition}

	\begin{definition} \label{epos} A symmetric function $X\in \Lambda$ is \em
		$e$-positive \normalfont if it has non-negative coefficients in the
		basis of the elementary symmetric functions.
	\end{definition}
	
	\begin{definition} \label{pfunc} Denote by $p_m$ the $m$-th power sum
		symmetric function: $$p_m =
		\sum\limits_{i\in\mathbb{N}}x^m_{i}.$$ Given a
		partition $\lambda = (\lambda_1\geq \lambda_2\geq...\geq\lambda_k)$, we 
		define the power sum
		symmetric function
		$p_{\lambda} = \prod\limits_{i=1}^k p_{\lambda_i}.$ These functions also
		form a basis of $\Lambda.$
	\end{definition}
	
	\begin{definition} \label{mfunc}
		Given a partition $\lambda = (\lambda_1\geq 
		\lambda_2\geq...\geq\lambda_k)$,  we define the monomial symmetric 
		function 
		$$m_\lambda=\sum\limits_{i_1<i_2<...<i_k}\sum\limits_{\lambda'\in 
		S_k(\lambda)}x_{i_1}^{\lambda_{1}'}\cdot
		x_{i_2}^{\lambda_{1}'}\cdot...\cdot x_{i_k}^{\lambda_{k}'},$$
		where the inner sum is taken over the set of all permutations of the 
		sequence $\lambda$, denoted by $S_k(\lambda)$.
	\end{definition}

	\begin{example}
		The chromatic symmetric function of $K_n$, the complete graph on $n$
		vertices, is  $X_{K_n} = n!\,e_n$; in particular, $X_{K_n}$ is $e$-positive.
	\end{example}
	
	\begin{definition} \label{incgraph} For a poset $P$, the \em incomparability
		graph\normalfont, $\textnormal{inc}(P)$, is the graph with elements
		of $P$ as vertices, where two vertices are connected if and only if
		they are not comparable in $P$.
	\end{definition} 
	
	\begin{definition} \label{nplusmfree} Given a pair of natural numbers
		$a,b\in\mathbb{N}^2$, we say that a poset $P$ is \em (a+b)-free
		\normalfont if it does not contain a length-$a$ and a length-$b$
		chain, whose elements are  incomparable.
	\end{definition} 
	
	\begin{definition} A unit interval order (UIO) is a partially ordered set
		which is isomorphic to a finite subset of $U\subset\R$ with the 
		following poset structure:
		\[ \text{for } u,w\in U:\    u\succ w \text{ iff } u\ge w+1.
		\]
		Thus $u$ and $w$ are incomparable precisely when $|u-w|<1$ and we will
		use the notation $u\sim w$ in this case. 
	\end{definition}
	\begin{theorem}[Scott-Suppes \cite{Scott-Suppes54}]\label{S_S}
		A finite poset $P$ is a UIO if and only if it is $(2+2)$- and 
		$(3+1)$-free.
	\end{theorem}

	Stanley~\cite{Stanley95a} initiated the study of incomparability
	graphs of $(3+1)$-free partially ordered sets. Analyzing the chromatic
	symmetric functions of these incomparability graphs,
	Stanley~\cite{Stanley95a} stated the following positivity conjecture.
	
	\begin{conjecture}[Stanley] \label{eposconj}
		If $P$ is a $(3+1)$-free poset, then $X_{\textnormal{inc}(P)}$ is 
		$e$-positive.
	\end{conjecture}
	
	For a graph $G$ let us denote by ${c_\lambda}(G)$ the coefficients of $X_G$ 
	with respect to the $e$-basis. We omit the index $G$ whenever this causes 
	no confusion. The conjecture thus states that in the expansion
	$$X_G=\sum\limits_{\lambda}c_{\lambda}e_\lambda,$$
for the case when $ G $ is the incomparability graph of a $(3+1)$-free poset, 
the coefficients $c_{\lambda}  $ are nonnegative for all partitions $ \lambda $.

	Conjecture~\ref{eposconj} has been verified with the help of computers
	for up to 20-element posets~\cite{Guay-Paquet13}. In 2013,
	Guay-Paquet~\cite{Guay-Paquet13} showed that to prove this conjecture,
	it would be sufficient to verify it for the case of $(3+1)$- and
	$(2+2)$-free posets, i.e. for unit interval orders (see Theorem~\ref{S_S}). 
	More precisely:
	
	\begin{theorem}[Guay-Paquet]\label{G_P}
		Let $P$ be a $(3+1)$-free poset. Then, $X_\mathrm{inc}(P)$ is a
		convex combination of the chromatic symmetric
		functions $$\{X_\mathrm{inc}(P')\ |\ P'\ \mathrm{is}\ \mathrm{a}\
		(3+1)\mathrm{-}\ \mathrm{and}\ (2+2)\mathrm{-free}\ \mathrm{poset}
		\}.$$
		
	\end{theorem}

	For a long time the strongest general result in this direction was that of 
	Gasharov
	\cite{Gasharov94}.
	
	\begin{definition} \label{sfunc} For a partition $\lambda = (\lambda_1\geq
		\lambda_2\geq...\geq\lambda_k)$, define the Schur
		functions \em $s_{\lambda}=\mathrm{det}(e_{\lambda_i^*+j-i})_{i,j}$,
		\normalfont where $\lambda^*$ is the conjugate partition to
		$\lambda$. The functions $\{ s_{\lambda}\}$ form a basis of
		$\Lambda$.
	\end{definition}

	\begin{definition} \label{spos}
		A symmetric polynomial $X$ is \em $s$-positive \normalfont if it has 
		non-negative coefficients in the basis of Schur functions.
	\end{definition}
	
	Obviously, a product of $e$-positive functions is $e$-positive. This also 
	holds for $s$-positive functions. Thus, the equality $e_n=s_{1^n}$ implies 
	that $e$-positive functions are $s$-positive, and thus $s$-positivity is 
	weaker than $e$-positivity.
	
	\begin{theorem}[Gasharov] \label{sposthm} 
		If $P$ is a $(3+1)$-free poset, then $X_{\textnormal{inc}(P)}$ is 
		$s$-positive.
	\end{theorem}
	
	Gasharov proved $s$-positivity by  constructing so-called $P$-tableau and 
	finding a one-to-one correspondence between these tableau and 
	$s$-coefficients~\cite{Gasharov94}. 

	The strongest known result on the $e$-coefficients was obtained by Stanley 
	in~\cite{Stanley95a}. He showed that sums of $e$-coefficients over the 
	partitions of fixed length are non-negative:

	\begin{theorem}[Stanley]
		For a finite graph $G$ and $j\in\mathbb{N}$, suppose 
		$$X_G=\sum\limits_{\lambda}c_{\lambda}e_\lambda,$$
		and let $\text{sink}(G,j)$ be the number of acyclic orientation of $G$ 
		with $j$ sinks. Then
		$$\text{sink}(G,j)=\sum\limits_{l(\lambda)=j}c_{\lambda}.$$
	\end{theorem}
	
	\begin{remark}
		By taking $j=1$, it follows from the theorem that $c_n$ is non-negative.
	\end{remark}

	In \cite{PSz}, we introduced a combinatorial device that we called   {\em correct sequences} ( or  {\em corrects}).
	
	\begin{definition}
		Let U be a UIO.  We will call a sequence $w = [w_1,\dots, w_k]$ of 
		elements of $U$ {\em
			correct} if
		\begin{itemize}
			\item 
			$w_i\not\succ w_{i+1}$ for $i=1,2,\dots,k-1$ 
			\item and for
			each $j=2,\dots,k$, there exists $i<j$ such that $w_i\not\prec w_j$.
		\end{itemize}
	\end{definition}
	Every sequence of length 1 is correct, and sequence $[w_1,w_2]$ is
	correct precisely when $w_1\sim w_2$. The second condition (supposing that 
	the first one holds) may be reformulated as follows:
	for each $j=1,\dots k$, the subset $\{w_1,\dots,w_j\}\subset U$ is connected
	with respect to the graph structure~${(U,\sim)}$.
	Using this notation, we proved the following.
	\begin{theorem}\label{eposn} 
		Let $X_{\text{inc}(U)}=\sum\limits_{\lambda}c_\lambda e_\lambda$ be a 
		chromatic symmetric function of the $n$-element unit interval order 
		$U$. Then $c_n$ is equal to the number of corrects of length $n$, in 
		which every element of $U$ is used exactly once.
	\end{theorem}

In particular, we showed that the coefficient $ c_n $ of the chromatic polynomial 
is non-negative for UIOs. This reproduced results by Stanley~\cite{Stanley95b} 
and Chow~\cite{Chow95}, who showed the positivity of 
	$c_n$ for $(3+1)$-free posets using combinatorial techniques, and linked 
	$e$-coefficients with the acyclic orientations of the incomparability 
	graphs. Using corrects allowed us to go further, and show the following
	(see~\cite{Paunov16b}):	
	\begin{theorem} \label{eposn21} 
		Let $X_{\text{inc}(P)}=\sum\limits_{\lambda}c_\lambda e_\lambda$ be a 
		chromatic symmetric function of the $(3+1)$-free poset~$P$, and 
		$k\in\mathbb{N}$. Then $c_{n-k,1^k}$, $c_{n-2,2}$, $c_{n-3,2,1}$ and 
		$c_{2^k,1^{n-2k}}$ are non-negative integers.
	\end{theorem}

In this article, we use an alternative combinatorial presentation of the coefficient $ c_n $ 
introduced by Siegl in \cite{Siegl}, where they also showed $ e $-positivity for all partitions
with equal parts. We took up this new tool, and managed to give a very sort, elegant proof of Stanley's conjecture for all partitions of length 2: 
$ \lambda=n\ge k $. This approach can be extended to the case of arbitrary partitions, 
but this work is  somewhat more involved. Very recently, another proof was announced for the length-2 case by Abreu and Nigro \cite{Abreu} 
using substantially more complicated tools which links this combinatorial problem to the topology of certain algebraic varieties. 
We hope that the combination of these approaches will lead to further insights into the meaning of e-positivity.

This short article is structured as follows: we begin with recalling Stanley's homomorphism, 
which we used in our previous work as well \S\ref{sec:Stan}. We set up our notations in \S\ref{sec:prepar}, 
and give some initial structural statements in the \S\ref{sec:structural}. The main construction and the proof is presented then in \S\ref{sec:proof}.

	\textbf {Acknowledgments.} We are grateful to Alexander Paunov for his help, ideas, and his generous support.  
	
	\section{Stanley's $G$-homomorphism}\label{sec:Stan}
Let $G$ be a finite graph 
	with vertex set $V(G)=\{v_1,...,v_n\}$ and edge set $E(G)$. 
We will consider the vertices as formal commuting variables and introduce the polynomial ring 
$\Lambda_G=\mathbb{R}[v_1,...,v_n]$. Stanley \cite[p.~6]{Stanley95b}  defined a ring homomorphism 
$ \rho_G:\Lambda\to\Lambda_G$ by setting the images of the free generators of $ \Lambda $, the elementary symmetric polynomials:
$$e_i^G =\sum\limits_{\substack{S\subset V\\
				 S\mathrm{-indep},\, \#S=i}}\prod\limits_{v\in S}v,$$ where the sum is
		taken over all $i$-element independent subsets $S$ of $V$, in which no two
		vertices form an edge, i.e. independent subsets. We set $e_0^G=1$, and 
		$e_i^G=0$ for $i<0$.

  The correspondence $e_i\mapsto e_i^G$ then extends to the
	ring homomorphism $ \rho_G:\Lambda\to\Lambda_G$, called Stanley's 	{\em $G$-homomorphism}. 
For $f\in \Lambda$, we will use the notation $f^G$ 
	for $\rho_G(f)$.
\begin{remark}
	Note that the polynomials $f^G$ are not necessarily symmetric.
	\end{remark}

	\begin{example}
		Given a partition $\lambda = \lambda_1\geq 
		\lambda_2\geq...\geq\lambda_k,\ k\in \mathbb{N},$ we have 
		$$e_{\lambda}^G = \prod\limits_{i=1}^k e_i^G,$$
		$$s_{\lambda}^G=\mathrm{det}(e_{\lambda_i^*+j-i}^G).$$
	\end{example}

	For an integer function $\alpha: V\rightarrow \mathbb{N}$ define the monomial
	$$v^\alpha = \prod\limits_{v\in V}v^{\alpha(v)},$$ 
	and introduce the notation $[v^\alpha]f^G$  for the coefficient of $v^\alpha$ in the 
	polynomial $f^G\in\Lambda_G$.

	Let $G^\alpha$ denote the graph, obtained by replacing every vertex $v$ of 
	$G$ by the complete subgraph of size $\alpha(v)$: $K_{\alpha(v)}^v$. Given 
	vertices $u$ and $v$ of $G$, a vertex of $K_{\alpha(v)}^v$ is connected to 
	a vertex of $K_{\alpha(u)}^u$ if and only if $u$ and $v$ form an edge in 
	$G$.

Stanley 
	\cite[p.~6]{Stanley95b} found the following connection between the $G$-analogues of 
	symmetric functions and the chromatic polynomial $X_G$. Following Stanley~\cite{Stanley95b}. 

Let $ x $ and $ y $ stand for two infinite set of variables. Using the Cauchy identity, one can prove the well-known identity
	$$\sum\limits_\lambda s_\lambda(x)s_{\lambda^*}(y)=\sum\limits_\lambda 
	m_\lambda(x)e_\lambda(y) = \sum\limits_\lambda e_\lambda(x)m_\lambda(y)$$

Applying the $ G $-homomorphism $ \rho_G $, we obtain the function
	\begin{equation}\label{GCauchy}
		T(x,v) = \sum\limits_\lambda m_\lambda(x)e^G_\lambda(v) = 
		\sum\limits_\lambda s_\lambda(x)s^G_{\lambda^*}(v)= 
		\sum\limits_\lambda e_\lambda(x)m^G_\lambda(v).
	\end{equation}
	where the sums are taken over all partitions. Then 
	
	\begin{equation}\label{gnechrom}
		[v^\alpha]T(x,v)\prod\limits_{v\in V}\alpha(v)! =X_{G^\alpha}.
	\end{equation}

	An immediate consequence of the formulas \eqref{GCauchy} and
	\eqref{gnechrom} is the following result of Stanley:
	\begin{theorem}[Stanley]\label{poscrit}
		For every finite graph G
		\begin{enumerate}
			\item $X_{G^\alpha}$ is s-positive for every
			$\alpha:V(G)\rightarrow\mathbb{N}$ if and only if $s_\lambda^G\in
			\mathbb{N}[V(G)]$ for every partition $\lambda$.
			\item $X_{G^\alpha}$ is e-positive for every 
			$\alpha:V(G)\rightarrow\mathbb{N}$ if and only if $m_\lambda^G\in 
			\mathbb{N}[V(G)]$ for every partition $\lambda$. 
		\end{enumerate}
	\end{theorem}
	
	\begin{remark}\label{c_m}
		If $X_{G^\alpha}=\sum\limits_{\lambda}c^\alpha_{\lambda}e_\lambda,$ 
		then $c_\lambda^\alpha=[v^\alpha]m^G_\lambda.$ Hence, monomial 
		positivity of $m^G_\lambda$ is equivalent to the positivity of 
		$c_\lambda^\alpha$ for every $\alpha$.
	\end{remark}

From now on, we will concentrate on the incomparability graphs of UIOs, and we will use the simplified notation $ m_\lambda^U $ for the relevant coefficient.
Then Theorem \ref{poscrit} combined with the reductions cited in \S\ref{sec:intro} leads to the following simple  reformulation of Stanley's chromatic conjecture in the language of $ G $-polynomials.
\begin{proposition}
Suppose that every UIO  $ U = \{v_1,\dots,v_n\} $, the coefficient of the 
monomial $ \prod_{j=1}^n v_j $ in the polynomial $ m^{\mathrm{inc}(U)}_\lambda $ is nonnegative. Then Conjecture 
\ref{eposconj} is true.
\end{proposition}

We have thus defined for each UIO $ U $ of $ n $ unit intervals on the real line and 
each partition $ \lambda  $ of $ n $ an integer $[\prod_{j=1}^n v_j] m^{\mathrm{inc}(U)}_\lambda $, which, from now on, we will simply denote by  $ m_\lambda^U $:
\[  m_\lambda^U  \overset{\mathrm{def}}{=}[\prod_{j=1}^n v_j] m^{\mathrm{inc}(U)}_\lambda \] 
Our task is to show that this number is nonnegative for all partitions. We will treat the case of length-2 partitions in this short paper.
	
\section{Preparations}\label{sec:prepar}

A key idea appeared in \cite{}, where an alternative definition of correct 
sequences was introduced. It will be convenient to set the following 
notation for sequences in unit interval orders. 

Given two unit intervals $ w_1,w_2 $ on the real line there is a trichotomy of 
possible relative position between them:
\begin{enumerate}
	\item $ w_1 $ intersects $ w_2 $,
	\item $ w_1\cap w_2=\emptyset $ and $ w_1  $ is left of $ w_2 $,
	\item $ w_1\cap w_2=\emptyset $ and $ w_1  $ is right of $ w_2 $.
\end{enumerate}
We will use the notation $ w_1\prec w_2 $ for case (2), $ w_1\succ w_2 $ for 
case 
(3) and $ w_1\ra w_2 $ if \textit{either case (1) or case (2) holds}. 

Let $ U $
 be a UIO, i.e  a set of unit intervals on the real line.
We will call a sequence of  elements of $ U $
\[ w_1\lhd w_2\lhd\dots \lhd w_{l-1}\lhd w_m,\]
where,  in each place, $ \lhd$ is one of 
the symbols $\prec $  or $ 
\ra $, a \textit{logical sequence}, which is, in fact, a statement (which is either true or not).

We will use the structure of UIOs in the following form then.
\begin{lemma}\label{funlemma}
	Let 
	\[ w_1\lhd w_2\lhd\dots \lhd w_{m-1}\lhd w_m\]
	be a logical sequence, which contains $ r $ $ \prec $-symbols, and $ s $ $ \ra 
	$-symbols with $ r+s=m-1 $ and $ r\ge s $. Then this statement is false.  
\end{lemma}
Informally, we will call such sequences \textit{impossible}.

We call a sequence $ [w_1,w_2,\dots,w_k] $ of different elements of $ U $ a \textit{$ k 
$-Escher} if
\[ w_1\ra w_2\ra\dots\ra w_k\ra w_1 \]
holds.
We will denote the operation of cyclic permutation on 
sequences by $ \zeta $:
\[    \zeta: [w_1,w_2,\dots,w_{k}]\mapsto[w_k,w_1,w_2\dots,w_{k-1}].
\]
Then acting by powers of $ \zeta$ on an $ k $-Escher produces $ k-1 $ new  $ k 
$-Eschers. As all the $k$ statements we obtain this way are equivalent,  
sometimes, we will use the term \textit{cyclic Escher} for the isomorphism 
class of sequences related by $ \zeta $. An Escher then is a cyclic  
Escher with a chosen initial point. Also, given a $ k $-Escher, we will think 
of the index set as integers reduced to modulo $ k $; thus $ w_0=w_k $, for 
example.

A key result of \cite{} is that the number $ m_n^U $
of a UIO is the same as the number of Eschers of the same length. Combined with our earlier results, this leads to the following statement.
\begin{theorem}\label{correctEscher}
	Let $ U = \{v_1,\dots,v_n\}$ be a UIO. Then 
	\[  m_n^U =\#\{\text{correct sequences of length }n\text{ in } U\}
	=\#\{\text{$ n $-Eschers in } U\}.
	\]
\end{theorem}
This is thus the first instance of the conjecture: it shows that statement for 
all "partitions" of length 1.
\begin{corollary}
	Since $ \#\{\text{$ n $-Eschers in } U\}=N\cdot\#\{\text{cyclic $ n
		$-Eschers in } U\} $, we can conclude that $ m_n^U $ is divisible by 
		$ n	$.
\end{corollary}

\section{The conjecture in the length-2 case} \label{sec:structural}

\subsection{A reformulation using Eschers} Fix to positive integers $ n\ge k $, and set $ N=n+k $. In this paper, we will show that 
$ m_{nk}^U\ge0 $ for any UIO $ U $ with $ N $ elements.

For a positive integer $ m $, introduce the notation 
\[    P^U_{m} = \{\text{length-}m\text{ Eschers in } U\}  \]
Let us suppose, for simplicity, that $ n>k $.
To study $ m_{nk}^U $, we will use the formula
\[   m_{nk} = p_n\cdot p_k - p_{n+k}\text{ in }\Lambda.\]
Applying Stanley's homomorphism, this implies that 
\[  m_{nk}^U  = \# P_{n}^U\cdot \# P_{k}^U-\#P^U_{n+k}. \]

Then the chromatic conjecture for length-2 partitions will follow from the 
following statement.
\begin{theorem} \label{thmmain}
	Let $ n>k $, and let $ U $ be a UIO of $ n+k  $ intervals. Then 
$$\# P_{n}^U\cdot \# P_{k}^U\ge\#P^U_{n+k}. $$
\end{theorem}
We will give a more precise statement below. For now, we begin with some 
preparatory work. 

\subsection{Sub-Eschers}  We  introduce  a key 
notion in this paragraph. 
\begin{lemma} \label{kEschers}
	Let $ U $ be a UIO of $ N=n+k $ elements, and let $ 
w=[w_1,\dots,w_{N}]\in P_N^U  $ be an 
	Escher in $ U $. 
 Then for any value of the index $ m $ (all indices are 
understood modulo $ N $), we have 3 mutually exclusive possibilities:
	\begin{enumerate}
		\item $ w_{m+k}\to w_{m+1} $ and $ w_m\to w_{m+k+1} $
		\item $ w_{m+k}\succ w_{m+1} $ and $ w_m\to w_{m+k+1} $
		\item $ w_{m+k}\to w_{m+1} $ and $ w_m\succ w_{m+k+1} $
	\end{enumerate}
\end{lemma}
\begin{proof}
	Indeed, the last remaining possibility: $ w_{m+k}\succ w_{m+1}  $ and $ 
	w_m\succ w_{m+k+1} $ would create the impossible sequence (cf. Lemma 
	\ref{funlemma}):
	\[ w_m\ra w_{m+1}\prec w_{m+k}\to w_{m+k+1}\prec w_m. \]
\end{proof}
Note that in cases  (1)  and (3)  the sequence $ [w_{m+1},\dots,w_{m+k}] $ is a $ k 
$-Escher, and in cases (1) and (2), the sequence $ [w_{m+k+1},\dots,w_m] $ 
is an $ n $-Escher. If case (1) holds, then we will call the sequence $ 
[w_{m+1},\dots,w_{m+k}] $ a \textit{valid $ k $-subEscher} of $ w $. 

The statement of this Lemma may be strengthened as follows.

\begin{lemma} \label{strengthened}
	Assume that we are in case (2) of Lemma \ref{kEschers}, and thus the 
	sequence $ [w_{m+1}, 
	w_{m+2},\dots, w_{m+k}] $ is \textbf{not} a $ k $-Escher. Then the three 
	sequences
	\begin{align*}
		&[w_{m+k}, w_{m+k+1},\dots, w_{m+k+n-1}], \\ &[w_{m+k+1}, 
		w_{m+k+2},\dots, w_{m+k+n}],\\ &[w_{m+k+2}, w_{m+k+3},\dots, 
		w_{m+k+n+1}]. 
	\end{align*}
	are all $ n $-Eschers in $ U $.
\end{lemma}
The proof is analogous to the one given above and will be omitted.

Now assume
\[ w_0\to w_1\to\dots\to w_{L-1}\to w_L. \]
We will call such a sequence \textit{pure} if for all $ 0<m\le L-k $, case (2) 
or (3) of Lemma \ref{kEschers} hold (i.e. there are no valid 
$ k $-subEschers).

Then Lemma \ref{strengthened} has the following important
\begin{corollary} \label{purecor}
	If $ w_0\to w_1\to\dots\to w_{L-1}\to w_L$ is pure than either case (2) 
	holds for all $ 0\le m< L-k $ or case (3) holds for all $ 0\le m<L-k $.
\end{corollary}
In the first case, we will call the sequence pure$^- $, while in the second we 
will call it pure$^+ $. These arguments have imply the following key statement.

\begin{proposition} \label{vsubEschers}
Let $ U $ be a UIO, and $ w\in P^U_N$ be an Escher.  Then there exists an index 
$ i  $ such that \\
 $ [w_{i+1}, w_{i+2},\dots, w_{i+k}] $ is a valid $ k 
$-Escher, i.e. that such that
$ [w_{i+k+1}, w_{i+k+2},\dots, w_{i+k+n}] $ is an Escher as well.
\end{proposition}
\begin{proof}
	Indeed, assuming the contrary, we would obtain an cyclic (infinite) 
	length-$ N $ pure sequence. If this sequence is \pum, then we obtain an 
	infinite logical
	sequence
	\[ w_0\prec w_{k-1}\prec w_{2(k-1)}\prec\dots, \]
	which is clearly impossible. The case of a \pup sequence is similar.
\end{proof}

\subsection{Insertions}
Now we pass to a dual notion. Let $ u\in P^U_n $ and $ v\in P^V_k $. We have a 
similar set of statements in this case.
\begin{lemma} \label{insertionlemma}
	For each index $ i\in\mathbb{N} $,  we have 3 mutually exclusive possibilities:
	\begin{enumerate}
		\item $ u_i\to v_{i+1} $ and $ v_i\to u_{i+1} $,
		\item $ u_i\succ v_{i+1} $ and $ v_i\to u_{i+1} $,
			\item $ u_i\to v_{i+1} $ and $ v_i\succ u_{i+1} $.
	\end{enumerate}
Moreover, between two subsequent indices of type (1), one has always only 
indices of same type, (2) or (3).
\end{lemma}
\begin{proof}
	The proof is similar to the one given above and will be omitted.
\end{proof}
Let us assume for simplicity that $ n $ and $ k $ are relatively prime. Then 
the pair $ (u,v ) $ gives rise to a cyclic sequence of $ n\cdot k $ pairs in $ 
P^U_n\times P^U_k $, related by simultaneous rotations,  all reducing to the 
same pair of cyclic Eschers, but identified by their starting points at $ u_i $ 
and $ v_i $, respectively, $ i=0,1,\dots,nk-1 $. We will call a pair of 
subsequent indices $ \kappa=(i,i+1) $ a \textit{valid insertion} if the case 
(1) above holds. Visually, we will think of this as
\[  \begin{pmatrix}
\dots 	v_i \\ \dots u_i 
\end{pmatrix}    \overset{\kappa}\times
 \begin{pmatrix}
 u_{i+1} \dots \\	v_{i+1} \dots
\end{pmatrix}
 \]
 and write $ i<\kappa< i+1 $.

\section{The proof} \label{sec:proof}

	We fix a UIO $ U $ of $ N=n+k $ elements. If this causes no confusion, then we will omit the index "$ U $" from our notation.  Following the ideas of \cite{PSz}, we will prove Theorem \ref{thmmain} 
	statement by 
	constructing an explicit injection $ \phi= \phi_{nk}:P^U_{n+k}\hookrightarrow 
	P_{n}^U\times P_{k}^U$. 

 Again, we will think of an element $ w\in P^U_m $ as a cyclic Escher with a 
 marked starting point indexed by 0. In particular, all indices are understood 
 modulo $ N=n+k $. One subtlety in our conventions: even though the indices are 
 arranged cyclically, we will still talk about \textit{a smallest} index 
 satisfying a certain condition, \textit{starting at some point}. This will 
 mean starting from a point, and then moving forward until we reach the index 
 with the required property.

\textbf{The map $ \phi: P_{n+k}\to P_n\times P_k$.} Assume for simplicity $ n>k 
$.

Then $ \phi $ is defined as follows: let $ w\in P^U_{n+k} $; define $0\le L\le 
n+k-1 $ to be the smallest index $ l $, starting from 0, such that
\[ [w_{l+1},\dots,w_{l+k}] \]
is a valid sub-Escher, i.e. $ [w_{l+k+1},\dots,w_{l}] $ is also an Escher. We 
will denote this index $ L $ by $ FE(w) $ sometimes
 We 
thus ended up with two, well-defined cyclic Eschers. All is needed then is to 
fix their initial points. We define
\begin{itemize}
	\item we set initial point of the $ k $-Escher as the index $ q\in[L+1,L+k] 
	$ such that $ q=0\mod k $, and similarly,
	\item we set initial point of the $ n $-Escher as the index $ q\in[L+k+1,L] 
	$ such that $ q=0\mod n $.
\end{itemize}
\begin{remark}
	Even though we have this uniform definition, there two distinct situations:
	\begin{enumerate}
		\item  The \textit{ordinary} case: $ 0\notin[L+1,L+k] $.  

		\item The \textit{exceptional} case :  $ 0\in[L+1,L+k] $. 
	\end{enumerate}
	
\end{remark}

 \textbf{The "inverse map" $ \psi: P_n\times P_k \to 
P_{n+k} $.} Let us take $ v $ and $ u $, a $ k $- and $ n $-Escher, 
respectively. Let $ L $ be the smallest integer $ 0\le l $ such that there is a 
valid 
insertion after $ l $, i.e. $ v_l\to u_{l+1} $ and $ 
u_l\to v_{l+1} $ (if there is no such $ l $, then we define the map in an 
arbitrary fashion). We will denote this index $ L $ of first valid insertion by 
$ FI(u,v) $ sometimes. This gives us immediately a cyclic $ n+k $-Escher, 
and we just need pick a initial point.
\begin{enumerate}
	\item The \textit{ordinary} case: $ L<n $. Then we take $ u_0 $ to be the 
	initial point.
	\item The \textit{exceptional} case: $ L\ge n $. Then we take $ v_n $ 
	to be 
	the initial point.
\end{enumerate}
The following theorem is a refinement of Theorem \ref{thmmain}, and immediately 
implies the chromatic conjecture for 
partitions of length 2.
\begin{theorem}
	\[ \psi\circ\phi\text{ is the identity} \]
\end{theorem}

Let us start with the ordinary case: $ w\in P_{n+k}$ and assume that the  
first $ k 
$-Escher occurs at $ FE(w)=L<n $. It is clear that the pair $ \phi(w)=(u,v) $ 
has a valid 
insertion right after the index $ L $, in principle, recovering $ w$. The 
danger is, 
however, that if $ FI(u,v)<L $, i.e. there are other valid insertions of the 
pair 
$ (u,v) $ before the index $ L $, then the map $ \psi $ will not take $ \phi(w) 
$ back to $ w $. We thus need to show that $ FI(\phi(w))\ge FE(w) $.
This is, essentially the content of the following statement:
\begin{proposition}\label{mainprop}
	Let $ \kappa_1 $ and $ \kappa_2 $ be two consecutive valid insertions of 
	the pair $ (u,v) $: 
	\[ j< \kappa_1< j+1,\dots,j+r<\kappa_2<j+r+1,\dots,\quad r\le n \] 
	Then the sequence 
	\[ u_j\to u_{j+1}\to\dots \to u_{j+r}\to v_{j+r+1}\to\dots \to v_{j+r+k} \]
	is not pure, i.e. has a valid $ k $-subEscher. 
\end{proposition}
\begin{proof}
	For the purposes of the proof, we can set $ j=0 $, as $ j $ simply 
	introduces a global shift in the arguments. Assume thus, ad absurdum that
	\[ u_0\to u_{1}\to\dots \to u_{r}\to v_{r+1}\to\dots \to v_{r+k} \]
	is pure. 
	
	Let us first consider the case when this sequence is \pum. Then by 
	definition of these \pum sequences, we have
	\[ u_1\prec u_k\prec\dots\prec u_{(q-1)(k-1)+1}\prec v_{q(k-1)+1}
	\to v_{q(k-1)+2} \to\dots\to v_{qk}=v_0\to u_1,
	\]
	which is an impossible sequence, because it has $ q $ $ \prec $ and $ q $ $ 
	\to $ symbols.
	If the sequence is \pup, then, on the other hand, by definition,
	\[ u_0\succ u_{k+1}\leftarrow u_k\succ\dots \leftarrow u_{(q-1)k}\succ 
	v_{qk+1}=v_1
	\leftarrow u_0, \]
	which is again an impossible sequence  with $ q $ $ \prec $ and $ q $ 
	$ \to $ symbols.
	
\end{proof}
We have thus shown that $ \psi\circ\phi(w)=w $ when $ 
FE(w)=L<n $.

The exceptional case is similar, but requires a little more care. Again, it is 
easy to see, that given an exceptional $ w\in P_{n+k}^W $, the pair $ 
(u,v)=\phi(w) $ has a valid insertion at $ L=FE(w)\ge n $, and the following 
statement will suffice.
\begin{proposition}\label{excepprop}
	Let $ \kappa_1 $ and $ \kappa_2 $ be two consecutive valid insertions of 
	the pair $ (u,v) $: 
	\[ 0< \kappa_1< 1,\dots,L<\kappa_2<L+1,\dots,\quad\text{ where } L\ge n. \] 
	Then the sequence 
	\[ v_n\to v_{n+1}\to\dots \to v_L\to u_{L-n+1}\to \dots \to u_L\to 
	v_{L+1}\to\dots  \to v_{L+k} \]
	is not pure, i.e. has a valid $ k $-subEscher. 
\end{proposition}
Note that we have set the starting index $ j=0 $, as in the proof of the case 
above.
\begin{proof}
	Assuming the sequence is \pum and $ L>n $, we obtain the following sequence:
	\begin{multline*}
		 v_{n+1}\prec u_k\prec\dots\prec u_{(q-1)(k-1)+1}\prec v_{q(k-1)+1}
	 \to\dots\\
	 \dots\to v_{qk}=v_0\to u_1=u_{n+1}\prec v_{n+k}=v_n\to v_{n+1},	
	\end{multline*}
which has with $ q+1 $ $ \prec $ and $ q+1 $ 
$ \to $ symbols.
When $ L=n $, the sequence degenerates to the one used in the proof of 
Proposition \ref{mainprop}.

Assuming the sequence is \pup, leads to the impossible sequence
\[ v_n\succ u_{k+1}\leftarrow u_k\succ\dots \leftarrow u_{(q-1)k}\succ 
v_{qk+1}=v_1
\leftarrow u_0 \succ v_{n+k+1}\leftarrow v_n, \]
and this completes the proof.
\end{proof}


\begin{thebibliography}{1}

\bibitem{Abreu} A. Abreu, A. Nigro, \emph{Splitting the cohomology of Hessenberg varieties and e-positivity of chromatic symmetric functions}, arXiv:2304.10644.

\bibitem{Chow95} T. Chow, \emph{A Note on a Combinatorial Interpretation of the e-Coefficients of the Chromatic Symmetric Function}, arXiv:math.co/9712230v2 (1995).


\bibitem{Gasharov94} V. Gasharov, \emph{Incomparability Graphs of
    (3+1)-free posets are s-positive}, Discrete Mathematics (1995), 157, 193-197.

\bibitem{Guay-Paquet13} M. Guay-Paquet, \emph{A modular relation for the chromatic symmetric functions of (3+1)-free posets}, arXiv:math.co/1306.2400 (2013).


\bibitem{Scott-Suppes54} D. Scott,  P. Suppes  \emph{Foundational aspects of theories of measurement}, Journal of Symbolic Logic (1954), 23, 113–128.



\bibitem{Paunov16} A. Paunov \emph{Positivity for Stanley's chromatic functions}, Genève University, (2016), available from:  http://archive-ouverte.unige.ch/unige:87600

\bibitem{Paunov16b} A. Paunov \emph{Planar graphs and Stanley’s Chromatic Functions}, preprint, (2016), arXiv:1702.05787 


\bibitem{PSz} A. Paunov, A. Szenes, \emph{A new approach to e-positivity for Stanley's chromatic functions}, 
arXiv:1702.05791

\bibitem{Siegl} I. Siegl, \emph{Cylindric P-Tableaux for (3+1)-Free Posets}, arXiv:2211.03953

\bibitem{Stanley95a} R. Stanley, \emph{A Symmetric Function
    Generalization of the Chromatic Polynomial of a Graph}, Advances
  in Mathematics (1995), 111, 166--194.


\bibitem{Stanley95b} R. Stanley, \emph{Graph colorings and related symmetric functions: Ideas and Applications}, MIT (1995).




\end{thebibliography}
\end{document}